\documentclass[11pt,a4paper]{article}
\usepackage[pdftex]{graphicx}
\usepackage{amsmath,amssymb,amsfonts,amsthm}
\numberwithin{equation}{section}
\usepackage{indentfirst}
\usepackage{enumitem} 
\usepackage[margin=1.3in]{geometry}
\usepackage{fancyhdr}

\usepackage[colorlinks=true,linkcolor=magenta,citecolor=blue,urlcolor=cyan]{hyperref}
\usepackage{titlesec}
\usepackage{etoolbox}
\usepackage{lettrine}
\usepackage{graphicx}  
\usepackage{changepage}
\theoremstyle{plain}
\newtheorem{theorem}{Theorem}[section]
\newtheorem{lemma}[theorem]{Lemma}

\newtheorem{proposition}[theorem]{Proposition}
\newtheorem{definition}[theorem]{Definition}
\newtheorem{remark}[theorem]{Remark}

\makeatletter
\renewcommand{\maketitle}{
	\begin{center}
		{\Large\bfseries{\@title}\par}
		\vskip 1em
		{\normalsize
			\lineskip .5em
			\begin{tabular}[t]{c}
				\@author
			\end{tabular}\par}
		\vskip 1.5em
	\end{center}
}
\makeatother
\renewenvironment{abstract}{
	\begin{adjustwidth}{1.3cm}{1.3cm}
		\noindent{\large\bfseries{A{\scriptsize BSTRACT.}}}
	}{
	\end{adjustwidth}
}
\usepackage{graphicx} 
\usepackage{xstring}  

\usepackage{xstring}
\usepackage{graphicx} 

\newcommand{\largerfirstscripts}[1]{%
	\StrLeft{#1}{1}[\firstletter]%
	\StrGobbleLeft{#1}{1}[\restofword]%
	{%
		\scalebox{0.96}{\MakeUppercase{\firstletter}}%
		{\scriptsize\MakeUppercase{\restofword}}%
	}%
}

\titleformat{\section}
{\normalfont\Large\bfseries\boldmath\filcenter}
{\thesection.}{0.5em}{\largerfirstscripts}

\usepackage{marvosym}

\begin{document}
	
	\title{ARITHMETIC PROPERTIES OF PARTITIONS WITH 1-COLORED EVEN PARTS AND r–COLORED ODD PARTS}
	
	\author{ M. P. Thejitha and S. N. Fathima}
	
	\maketitle
	
	\begin{abstract}
Recently, Hirschhorn and Sellers defined the partition function $a_r(n)$, which counts the number of partitions of $n$ wherein even parts come in only one color, while the odd parts may appear in one of $r$-colors for fixed $r\ge1$. The aim of this paper is to prove several new infinite families of congruences modulo 3 and 5 by employing a result of Newman and theory of modular forms. \\

		
	\end{abstract}
	\vspace{0.5em}
	\section{INTRODUCTION}	
A partition of a positive integer $n$ is a finite non-increasing sequence of positive integers $\lambda_1\ge\lambda_2\ge...\ge\lambda_r\ge1$	whose sum equals $n$ (see {{\cite[A000041]{onen}}}). Each $\lambda
_i$ is called a part of the partition. Let $p(n)$ denote the number of partitions of $n$, with the usual convention that $p(0)=1$. For example, the seven partitions of 5 are 
\begin{align*}
	(5), (4,1), (3,2), (3,1,1), (2,2,1), (2,1,1,1), (1,1,1,1,1).
\end{align*}
Hence, $p(5)=7$. The generating function of $p(n)$ is given by (see \cite{andr}),
\begin{align*}
	\sum_{n=0}^{\infty}p(n)q^n=\dfrac{1}{(q;q)_\infty}\; ,
\end{align*}
where the standard $q$-Pochhammer symbol $(a;q)_\infty$ is given by
\begin{align*}
	(a;q)_\infty = \prod_{k=0}^\infty(1-aq^{k}).
\end{align*}
We assume that $a$ and $q$ are complex numbers with $|q|<1$. Throughout this paper, for any integer $k\ge1$, we set
\begin{align*}
	f_k:= (q^k;q^k)_\infty \;.
\end{align*}
\indent Recently, Amdeberhan and Merca \cite{am1} considered the colored partition function $a(n)$ which counts the number of partitions of $n$ in which the odd parts is assigned one of three colors (see {{\cite[8298311]{onen}}}). For example, $a(2)=7$, since the partitions in question are: 
\begin{align*}
	(2), (1_1,1_1), (1_2,1_2), (1_3,1_3), (1_1,1_2), (1_1,1_3), (1_2,1_3).
\end{align*}
The generating function for $a(n)$ is given by 
\begin{align*}
	\sum_{n=0}^{\infty}a(n)q^n=\dfrac{f_2^2}{f_1^3}\; .
\end{align*}

Further, Amdeberhan and Merca have used mathematica package RaduRK developed by Smoot \cite{sm1} to prove the following generating function for $a(7n+2)$.

\begin{theorem}[{{\cite[Theorem~6]{am1}}}]\label{t 1.1} For $|q|<1$, we have
	\begin{align*}
		\sum_{n=0}^{\infty}a(7n+2)q^n&=7\displaystyle\left(\frac{1024f_2^8f_{14}^{18}}{f_1^{20}f_7^7}q^8+\frac{1344f_2^9f_{14}^{11}}{f_1^{21}}q^6-\frac{1024f_2^{16}f_{14}^{10}}{f_1^{24}f_7^3}q^5+\frac{72f_2^{10}f_7^7f_{14}^4}{f_1^{22}}q^4\right.\nonumber \\
		&\quad\left.-\frac{320f_2^{17}f_7^4f_{14}^3}{f_1^{25}}q^3-\frac{40 f_2^{11}f_7^{14}}{f_1^{23}f_{14}^3}q^2+\frac{56 f_2^{18}f_7^{11}}{f_1^{26}f_{14}^4}q+\frac{f_2^{12}f_7^{21}}{f_1^{24}f_{14}^{10}}\right).
	\end{align*}
\end{theorem}
\noindent As a corollary of Theorem \ref{t 1.1}, they have obtained the following arithmetic congruence modulo 7
\begin{align*}
	a(7n+2)\equiv 0\pmod 7	.
\end{align*}
\indent Very recently, Hirschhorn and Sellers \cite{hs1} have considered the generalization of partition function $a(n)$. For $r\ge1$, they defined $a_r(n)$ that counts the partition function of $n$ in which odd parts is assigned one of $r$ colours. The generating function of $a_r(n)$ is given by
\begin{align}\label{e1.1}
	\sum_{n=0}^{\infty}a_r(n)q^n=\dfrac{f_2^{r-1}}{f_1^r}\; .\end{align}
Clearly, $a_1(n)=p(n)$, $a_2(n)=\bar{p}(n)$ (the number of overpartition of $n$ \cite{cl1}, {{\cite[A015128]{onen}}}) and $a_3(n)=a(n)$.
Also, Hirchhorn and Sellers \cite{hs1} have employed theta function identities and $q$- series manipulation to prove arithmetic congruences modulo 7 for $a_k(n)$. Their results are stated in the following theorem.
\begin{theorem}[{{\cite[Theorem~1.3]{hs1}}}]
	For all $n\ge0$,
	\begin{align*}
		a_1(7n+5)\equiv 0 \pmod{7}\\
		a_3(7n+2)\equiv 0 \pmod{7}\\
		a_4(7n+4)\equiv 0 \pmod{7}\\
		a_5(7n+6)\equiv 0 \pmod{7}\\
		a_7(7n+3)\equiv 0 \pmod{7}.
	\end{align*}
	
\end{theorem}

The main purpose of this paper is to extend their lists and prove several more congruences for $a_s(n)$ modulo 3 and 5. We first collect generating functions, basic definitions, and recall some facts on modular forms. Our paper is of two folds. Firstly, in Section \ref{s3} and \ref{s4}, by employing a theorem of Newman \cite{new62}, we provide proofs of congruences modulo 5 for $a_3(n)$ and congruences modulo 3 for $a_t(n)$ for $t\in\{5,8,11,14,17,20,23,26\}$, respectively. Secondly, in Section \ref{s5}, we conclude by further obtaining congruences  modulo 3 and 5 for $a_5(n)$ using Sturm's result.

	\section{Preliminaries}
In this section, we begin with the following definition of Legendre symbol denoted by $\displaystyle\left(\frac{n}{p}\right)_L$ . 
\begin{definition}
	Let $n$ be an integer and $p$ be an odd prime. The Legendre symbol is defined by
	\[
	\left( \frac{a}{p} \right)_L :=
	\begin{cases}
		\hphantom{-}1, & \text{if } n \text{ is a quadratic residue modulo}\;p \;\text{and} \;p\nmid n,\\[4pt]
		\hphantom{-}0, & \text{if } p \mid n,\\[4pt]
		-1, & \text{if } n \text{ is a non-quadratic residue modulo } p.
	\end{cases}
	\]
\end{definition}
We now note the following pivotal result of Newman \cite{new62}.	
\begin{lemma}[{{\cite[Theorem~3]{new62}}}]\label{l2.2}
	Let $p$ and $q$ be distinct primes. The values of integers $r$ and $s$ such that $rs\not=0$ and $r\not\equiv s\pmod2$ are those given in the table \cite{new62}, where with the entry $(r,s)$ we must also include $(s,r)$. We set
	\begin{align*}
		\phi(\tau)&=\prod_{n=1}^{\infty}(1-x^n)^r(1-x^{nq})^s=\sum_{n=0}^{\infty}c(n)x^n,\\
		(\epsilon, t)&= \displaystyle\left(\dfrac{(r+s)}{2}, \dfrac{(r+sq)}{24}\right).
	\end{align*}
	For these values of $r,s,$ and $q$, the coefficients $c(n)$ of $\phi(\tau)$ satisfy
	\begin{align*}
		c(np^2+\Delta)-\gamma c(n)+p^{2\epsilon-2}c\displaystyle\left(\dfrac{n-\Delta}{p^2}\right)=0\;,
	\end{align*}
	where 
	\begin{align*}
		\gamma=p^{2\epsilon-2}\alpha-\displaystyle\left(\frac{\theta}{p}\right)_L p^{\epsilon-3/2}\displaystyle\left(\frac{n-\Delta}{p}\right)_L,
	\end{align*}
	$\theta=(-1)^{1/2-\epsilon} 2q^s$, $\Delta=t(p^2-1)$, and
	$\alpha$ is a constant.
\end{lemma}
 \noindent Our proofs in Section 3 and 4 relies heavily on Lemma \ref{l2.2}.\\
\indent We now recall some basic definitions and facts on modular forms. For more details, we refer the reader to \cite{knop}, \cite{kmf}, \cite{owom}, \cite{rmf}.\\
We start with definition of upper half of the complex plane $\mathbb{C}$. Let $\mathbb{H}=\{z: Im(z)>0\}$.\\
 \indent A subgroup $\Gamma$ of $SL_2(\mathbb{Z})$ is called a congruence subgroup if $\Gamma(N)\subseteq \Gamma$ for some $N$, and the smallest $N$ with this property is called the level of $\Gamma$. For a positive integer $N$, the congruence subgroup $\Gamma_0(N)$ of $SL_2(\mathbb{Z})$ is defined by
\begin{align*}
	\Gamma_0(N)&:= 
	\left\{
	\begin{bmatrix}
		a & b \\
		c & d
	\end{bmatrix} \in SL_2(\mathbb{Z})
	\,:\, c \equiv 0\!\!\!\! \pmod N 
	\right\}.
\end{align*}
We note that $\Gamma_0(N)$ is a congruence subgroup of level $N$, and $\gamma=\begin{bmatrix}
	a & b \\
	c & d
\end{bmatrix}$ acts on $\mathbb{H}$ by bilinear transformation defined by $\gamma z:=\dfrac{az
	+b}{cz+d} $.\\
\indent If $f(z)$ is a function on $\mathbb{H}$, which satisfies $f(\gamma z)=\chi(d)(cz+d)^kf(z)$, where $\chi$ is a Dirichlet character modulo 	$N$, and $f(z)$ is holomorphic on $\mathbb{H}$ and at all the cusps of $\Gamma_0(N)$, then we call $f(z)$ a holomorphic modular form of weight $k$ with respect to $\Gamma_0(N)$ and character $\chi$. The set of all holomorphic modular forms of weight $k$ with respect to $\Gamma_0(N)$ and character $\chi$ is denoted by $M_k(\Gamma_0(N), \chi)$.\\
\indent The relevant modular forms for the results of this paper arise from eta-quotients and Eienstein series of weight 4. Throughout, let $q=e^{2\pi iz}$.\\
The Dedekind's eta-function $\eta(z)$ is the simplest Euler product, and non-vanishing holomorphic function on $\mathbb{H}$ given by
\begin{align}\label{2.1}
	\eta(z):=q^{1/24}(q;q)_\infty=q^{1/24}\prod_{n=0}^\infty(1-q^{n}).
\end{align}
We note that \eqref{2.1} is a half-weight modular form with a certain multiplier on $SL_2(\mathbb{Z})$.
The relevant modular forms for the results obtained in our paper arise from eta-quotients. A function is called an eta-quotient if it is of the form
\begin{align*}
	f(z)=\prod_{\delta\mid N}\eta(\delta z)^{r_\delta},	
\end{align*}
where $N$, $\delta$ are positive integers and $r_\delta$ is an integer. Also, recall the classical Eisenstein series
\begin{align*}
	E_4(z)=1+240	\sum_{n=1}^{\infty}\sigma_3(n)q^n,	
\end{align*}
where $\sigma_3(n):=\sum_{d>0,\;d\mid n}d^3$. Now, we recall an important result due to Gordon-Hughes \cite{gh}, Newman \cite{new 2}, and
Ligozat \cite{lig} which is useful to verify whether an eta-quotient is a modular form. If $f(z)$ is an eta-quotient associated with positive integer weight $k$ satisfying \eqref{2.3}, \eqref{2.4} and is holomorphic at all the cusps of $\Gamma_0(N)$, then $f(z) \in M_k(\Gamma_0(N), \chi )$. The condition \eqref{2.5} gives the criteria to confirm that $f(z)$ has non-negative orders at the cusps. 
\begin{theorem}[{{\cite[Theorem~1.64 and 1.65]{owom}}}]\label{t2.1}
	If $ f(z)= \prod_{\delta \mid N} \eta(\delta z)^{r_\delta}$ is an eta-quotient with
	$k= \frac{1}{2} \sum_{\delta \mid N}{r_\delta} \in \mathbb{Z}$, and satisfies the following additional properties:
	\begin{align}\label{2.3}
		\sum_{\delta\mid N} \delta {r_\delta} \equiv 0 \pmod {24},
	\end{align}
	\begin{align}\label{2.4}
		\sum_{\delta \mid N} \frac{N}{\delta}  {r_\delta} \equiv 0 \pmod {24},
	\end{align}
	\begin{align}\label{2.5}
		\sum_{\delta\mid N} \frac{gcd(d,\delta)^2 r_\delta}{ \delta} \ge 0,		
	\end{align}
	then $f(z)\in M_k(\Gamma_0(N), \chi)$ where the character $\chi$ is defined by
	\begin{align*}
		\chi (d) := \bigg( \frac{(-1)^k \prod_{\delta \mid N} \delta^{r_\delta}}{d} \bigg).
	\end{align*}
\end{theorem}
We also require a result of Sturm \cite{sb}, which gives  criteria to determine whether two modular forms with integer coefficients are congruent modulo a given prime.
Let $m$ be a positive integer and $f(z)=\sum_{n=0}^{\infty}a(n)q^n$ with integer coefficients. The $ord_m
(f(z))$ is defined by  $ord_m
(f(z))=inf\{n\mid c(n)\not \equiv0 \pmod m\}$.
\begin{theorem}[{{\cite[Theorem~2.58]{owom}}}]\label{t2.2}
	Suppose p be a prime and $f(z)=\sum_{n=0}^{\infty}a(n)q^n$ and $g(z)=\sum_{n=0}^{\infty}b(n)q^n$ such that $f(z)$, $g(z) \in M_k(\Gamma_0(N), \chi)\cap \mathbb{Z}[[q]]$. If 
	\begin{align*}
		ord_p(f(z)-g(z))>\dfrac{kN}{12}\prod_{t\;prime, \;t\mid N}\bigg(\frac{1+t}{t}\bigg),
	\end{align*}
	then $ord_p(f(z)-g(z))=\infty$, $($which implies $a(n)\equiv b(n)\pmod p$$)$. 
\end{theorem}
Hecke  operators are linear transformations which act on spaces of modular forms and play a crucial role in our proof. We recall the definition of the Hecke operators on spaces of integer weight modular forms.\\
\begin{definition}[{{\cite[Definition~2.1]{owom}}}]
	Let $m$ be a positive integer and $f(z) = \sum_{n=0}^ \infty a(n)q^n \in  M_k(\Gamma_0(N), \chi ).$ The Hecke operator $T_m$ acts on $f(z)$ by 
	\begin{align}
		f(z)\mid {T_m} := \sum_{n=0}^\infty \bigg( \sum_{d\mid gcd(n,m)} \chi (d) d^{k-1}a \bigg(\frac{nm}{d^2} \bigg) \bigg)q^n.
	\end{align}
	In particular, if $m=p$ is a prime, then 
	\begin{align}\label{2.7}
		f(z)\mid {T_p} := \sum_{n=0}^\infty \bigg( a(pn)+ \chi (p) p^{k-1}a \bigg(\frac{n}{p} \bigg) \bigg)q^n.
	\end{align}
	This implies that, if $k>0$, then
	\begin{align*}
		f(z)\mid {T_p}\equiv \sum_{n=0}^\infty a(pn)\pmod p.
	\end{align*}
	We adopt the convention that $a(n/p)=0$ whenever $p\nmid n$.
\end{definition}

The following proposition is a direct consequence of \eqref{2.7} and definition of Hecke operator. 
\begin{proposition}\label{p2.4}
	Suppose that $p$ is prime, $f(z)\in\mathbb{Z}[[q]]$, $g(z)\in\mathbb{Z}[[q^p]]$, and $k>1$. Then 
	\begin{align*}
		(fg\mid T_{p,k,\chi})(z) \equiv (f\mid T_{p,k,\chi})(z)\cdot g(z/p)\pmod p.	
	\end{align*}
\end{proposition}
\begin{proposition}[{{\cite[Proposition~2.3]{owom}}}]\label{p2.5}
	Suppose that
	\begin{align*}
		f(z)=	\sum_{n=0}^{\infty}a(n)q^n\in M_k(\Gamma_0(N), \chi).
	\end{align*}
	If $m\ge2$, then $f(z)\mid T_{m,k,\chi} \in M_k(\Gamma_0(N), \chi)$. 
\end{proposition}
In what follows are the generating functions for $a_s(n)$ modulo 3 and 5, which will be used in subsequent sections \ref{s3} and \ref{s4}.
\begin{lemma}[{{\cite{rg1}}}] For all $n\ge0$,
		\begin{align}\label{e3.1}
		\sum_{n=0}^\infty a_3(5n+1)q^n \equiv 3f_1f_2^2\pmod 5.
	\end{align} 
\end{lemma}
\begin{lemma}[{{Sellers \cite{sellers}}}]
For all $n\ge0$,
		\begin{align}
		\sum_{n=0}^{\infty}a_5(3n+1)q^n&\equiv2f_1f_2^4 \pmod 3\label{e2.1}\\
		\sum_{n=0}^{\infty}a_8(3n)q^n&\equiv\dfrac{f_1^8}{f_2^3}\pmod 3\label{e2.2}\\
		\sum_{n=0}^{\infty}a_8(9n+2)q^n&\equiv\dfrac{f_1^8}{f_2}\pmod 3\label{e2.3}\\
		\sum_{n=0}^{\infty}a_{11}(3n)q^n&\equiv\dfrac{f_1^7}{f_2^2}\pmod 3\label{e2.4}\\
		\sum_{n=0}^{\infty}a_{11}(3n+1)q^n&\equiv 2\dfrac{f_2^6}{f_1}\pmod 3\label{e2.5}\\
		\sum_{n=0}^{\infty}a_{14}(3n)q^n&\equiv \dfrac{f_1^6}{f_2}\pmod 3\label{e2.6}\\
		\sum_{n=0}^{\infty}a_{14}(3n+1)q^n&\equiv 2\dfrac{f_2^7}{f_1^2}\pmod 3\label{e2.7}\\
		\sum_{n=0}^{\infty}a_{17}(3n+1)q^n&\equiv 2\dfrac{f_2^8}{f_1^3}\pmod 3\label{e2.8}\\
		\sum_{n=0}^{\infty}a_{17}(9n+5)q^n&\equiv 2\dfrac{f_2^8}{f_1}\pmod 3\label{e2.9}\\
		\sum_{n=0}^{\infty}a_{20}(3n)q^n&\equiv f_1^4f_2\pmod 3 \label{e2.10}\\
		\sum_{n=0}^{\infty}a_{23}(3n+1)q^n&\equiv 2\dfrac{f_2^{10}}{f_1^5}\pmod 3\label{e2.11}\\
		\sum_{n=0}^{\infty}a_{26}(3n)q^n&\equiv f_1^2f_2^3\pmod 3 \label{e2.12}.
	\end{align}
	
\end{lemma}

	\section{Congruences modulo 5 for $a_3(n)$}\label{s3}
	With the above preliminaries in place, we are now ready to prove the following theorem. 
	\begin{theorem}\label{t1.3}
		Let $c(n)$ be defined by
		\begin{align}\label{e1.2}
			\sum_{n=0}^{\infty}c(n)q^n:= f_1 f_2^2,
		\end{align}
		and $p\ge 5$ be a prime. Define
		\begin{align}\label{e1.3}
			\xi(p):= c\displaystyle\left(\frac{5(p^2-1)}{24}\right) + \displaystyle\left(\frac{\frac{5}{12}(p^2-1)}{p}\right)_L.
		\end{align}
		(1). For $n,k\ge0$, if $p\nmid n$, then
			\begin{align}\label{e1.4}
			a_3\displaystyle\left(5p^{\omega(p)(k+1)-1}n + \dfrac{25p^{\omega(p)(k+1)}-1}{24}\right) \equiv0\pmod 5,
		\end{align}
where
		\begin{align}\label{e1.5}
			\omega(p) :=\begin{cases}
				\hphantom{1}4, & \text{if } \xi(p) \equiv 0\!\!\!\!\pmod 5,\\[4pt]
				\hphantom{1}6, & \text{if }\xi(p) \equiv \pm 1\!\!\!\!\pmod5 \;\text{and}\; p\equiv1\!\!\!\!\pmod5, \;\text{or} \\
				& \quad\xi(p) \equiv \pm2\!\!\!\!\pmod5 \;\text{and}\; p\equiv4\!\!\!\!\pmod5,\\[4pt]
				\hphantom{1}8, & \text{if } \xi(p) \equiv\pm2\!\!\!\!\pmod5 \;\text{and}\;p\equiv2\!\!\!\!\pmod5,\; \text{or}\\& \quad\xi(p) \equiv\pm1\!\!\!\!\pmod5\;\text{and}\; p\equiv3\!\!\!\!\pmod5,\\[4pt]
				10, & \text{if } \xi(p) \equiv\pm 2\!\!\!\!\pmod5\;\text{and}\; p\equiv1\!\!\!\!\pmod5,\;\text{or} \\& \quad\xi(p) \equiv\pm 1\!\!\!\!\pmod5\;\text{and}\; p\equiv4\!\!\!\!\pmod5,\\[4pt]
				12, & \text{if } \xi(p) \equiv\pm 1\!\!\!\!\pmod5\;\text{and}\; p\equiv2\!\!\!\!\pmod5,\; \text{or} \\& \quad\xi(p) \equiv\pm 2\!\!\!\!\pmod5\;\text{and}\; p\equiv3\!\!\!\!\pmod5.\\[4pt]
			\end{cases}
		\end{align}
		(2). If $\xi(p)\not\equiv0\pmod5$, then for $n,k\ge0$ with $\xi(p)\equiv \displaystyle\left(\frac{-2n+\frac{5}{12}(p^2-1)}{p}\right)_L\pmod 5$,
		\begin{align}\label{e1.6}
			a_3\displaystyle\left(5p^{\omega(p)k+2}n + \dfrac{25p^{\omega(p)k+2}-1}{24}\right) \equiv0\pmod 5,
		\end{align}
		where $\omega(p)$ is defined by	\eqref{e 1.5}.\\[4pt]
		(3). For $n,k\ge0$, if $p=5$, then
		\begin{align}\label{e1.7}
			a_3\displaystyle\left(25\cdot5^{2(k+1)}n+\frac{25\cdot5^{2(k+2)}-1)}{24}\right)\equiv (2)^{k+1}a_3(25n+26)\pmod5.
		\end{align}
	\end{theorem}
	\begin{proof}[Proof of (1) of Theorem \ref{t1.3}\unskip]
		Thanks to Lemma \ref{l2.2}, from \eqref{e1.2}, we have
		\begin{align}\label{e3.2}
			c\displaystyle\left(p^2n+\dfrac{5(p^2-1)}{24}\right)=\gamma(n)c(n)-p \;c\displaystyle\left(\frac{n-\frac{5(p^2-1)}{24}}{p^2}\right),
		\end{align}
		where
		\begin{align}\label{e3.3}
			\gamma(n)=p\alpha-\displaystyle\left(\frac{-2n+\frac{5(p^2-1)}{12}}{p}\right)_L.
		\end{align}
	 Setting $n=0$ in \eqref{e3.2}, we get 
		\begin{align}\label{e3.4}
			c\displaystyle\left(\frac{5(p^2-1)}{24}\right)=\gamma(0).
		\end{align}
		Now, using \eqref{e3.4} in \eqref{e3.3} with $n=0$, we get
		\begin{align}\label{e3.5}
			p\alpha=c\displaystyle\left(\frac{5(p^2-1)}{24}\right)+\displaystyle\left(\frac{\frac{5(p^2-1)}{12}}{p}\right)_L:=\xi(p).
		\end{align}
			We now rewrite \eqref{e3.2}, with the help of  \eqref{e3.3} and \eqref{e3.5} as follows
		\begin{align}\label{e3.6}
			c\displaystyle\left(p^2n+\dfrac{5(p^2-1)}{24}\right)=\displaystyle\left(\xi(p)-\displaystyle\left(\frac{-2n+\frac{5(p^2-1)}{12}}{p}\right)_L\right)c(n)-p\;c\displaystyle\left(\frac{n-\frac{5(p^2-1)}{24}}{p^2}\right).
		\end{align}
		Replacing $n$ by $pn+\frac{5(p^2-1)}{24}$ in \eqref{e3.6}, we obtain
		\begin{align}\label{e3.7}
			c\displaystyle\left(p^3n+\frac{5(p^4-1)}{24}\right)=\xi(p)c\displaystyle\left(pn+\frac{5(p^2-1)}{24}\right)-p\;c\displaystyle\left(\frac{n}{p}\right).
		\end{align}
From \eqref{e3.7}, if $\xi(p)\equiv0\pmod5$, we know
		\begin{align}\label{e3.8}
			c\displaystyle\left(p^3n+\frac{5(p^4-1)}{24}\right)\equiv -p\;c\displaystyle\left(\frac{n}{p}\right)\pmod5.
		\end{align}
		Replacing $n$ by $pn$ in \eqref{e3.8}, gives for every integer $k\ge0$,
		\begin{align}\label{e3.9}
			c\displaystyle\left(p^4n+\frac{5(p^4-1)}{24}\right)\equiv-p\;c(n)\pmod5.	
		\end{align}
		By induction on \eqref{e3.9}, we see for every $k\ge0$,
		\begin{align}\label{e3.10}
			c\displaystyle\left(p^{4k}n+\frac{5(p^{4k}-1)}{24}\right)\equiv(-p)^kc(n)\pmod5.	
		\end{align}
		Also, for $p\nmid n$, \eqref{e3.8} yields
		\begin{align}\label{e3.11}
			c\displaystyle\left(p^3n+\frac{5(p^4-1)}{24}\right)\equiv0\pmod5.
		\end{align}
		Again, employing \eqref{e3.11} in \eqref{e3.10} with $n$ replaced by $p^3n+\frac{5(p^4-1)}{24}$, we see that 
		\begin{align}\label{e3.12}
			c\displaystyle\left(p^{4k+3}n+\frac{5(p^{4k+4}-1)}{24}\right)\equiv0\pmod5.	
		\end{align}
		Replacing $n$ by $p^2n+\frac{5p(p^2-1)}{24}$ in \eqref{e3.7}, yields
		\begin{align}\label{e3.13}
			c\displaystyle\left(p^5n+\frac{5(p^6-1)}{24}\right)=\xi(p)c\displaystyle\left(p^3n+\frac{5(p^4-1)}{24}\right)-p\;c\displaystyle\left(pn+\frac{5(p^2-1)}{24}\right).
		\end{align}
		Now, \eqref{e3.7} and \eqref{e3.13}, allows us to see that
		\begin{align}\label{e3.14}
			c\displaystyle\left(p^5n+\frac{5(p^6-1)}{24}\right)=(\xi^2(p)-p)c\displaystyle\left(pn+\frac{5(p^2-1)}{24}\right)-\xi(p)p\;c\displaystyle\left(\frac{n}{p}\right).
		\end{align}
		We observe that for $p\equiv1\pmod5$ and $\xi(p) \equiv \pm1\pmod5$, or $p\equiv4\pmod5$  and $\xi(p) \equiv \pm2\pmod5$, \eqref{e3.14} gives
		\begin{align}\label{e3.15}
			c\displaystyle\left(p^5n+\frac{5(p^6-1)}{24}\right)\equiv-\xi(p)p\;c\displaystyle\left(\frac{n}{p}\right) \pmod 5.
		\end{align}
		If $p\nmid n$, we have
		\begin{align}\label{e3.16}
			c\displaystyle\left(p^5n+\frac{5(p^6-1)}{24}\right)\equiv0\pmod 5.
		\end{align}
		Replacing $n$ by $pn$ in \eqref{e3.15}, we get
		\begin{align}\label{e3.17}
			c\displaystyle\left(p^6n+\frac{5(p^6-1)}{24}\right)\equiv-\xi(p)p\;c(n) \pmod 5.
		\end{align}
		By induction on \eqref{e3.17}, we obtain for every integer $k\ge0$,
		\begin{align}\label{e3.18}
			c\displaystyle\left(p^{6k}n+\frac{5(p^{6k}-1)}{24}\right)\equiv(-\xi(p)p)^kc(n) \pmod 5.
		\end{align}
		Using \eqref{e3.16} in \eqref{e3.18} with $n$ replaced $p^5n+\frac{5(p^6-1)}{24}$, yields
		\begin{align}\label{e3.19}
			c\displaystyle\left(p^{6k+5}n+\frac{5(p^{6k+6}-1)}{24}\right)\equiv0\pmod 5.
		\end{align}
			Again, replacing $n$ by $p^2n+\frac{5p(p^2-1)}{24}$ in \eqref{e3.13}, we obtain
		\begin{align}\label{e3.20}
			c\displaystyle\left(p^7n+\frac{5(p^8-1)}{24}\right)=\xi(p)c\displaystyle\left(p^5n+\frac{5(p^6-1)}{24}\right)-p\;c\displaystyle\left(p^3n+\frac{5(p^4-1)}{24}\right).
		\end{align}
	Now, we rewrite \eqref{e3.20} using \eqref{e3.7} and \eqref{e3.14} as follows
		\begin{align}\label{e3.21}
			c\displaystyle\left(p^7n+\frac{5(p^8-1)}{24}\right)=\xi(p)(\xi^2(p)-2p)c\displaystyle\left(pn+\frac{5(p^2-1)}{24}\right)+(p^2-p\;\xi^2(p))\;c\displaystyle\left(\frac{n}{p}\right).
		\end{align}
	We note, for $p\equiv2\pmod5$ and  $\xi(p) \equiv \pm2\pmod5$, or $p\equiv3\pmod5$ and $\xi(p) \equiv \pm1\pmod5$, \eqref{e3.21} becomes
		\begin{align}\label{e3.22}
			c\displaystyle\left(p^7n+\frac{5(p^8-1)}{24}\right)\equiv c\displaystyle\left(\frac{n}{p}\right)\pmod5.
		\end{align}
	For $p\nmid n$, it is easy to agree that
		\begin{align}\label{e3.23}
			c\displaystyle\left(p^7n+\frac{5(p^8-1)}{24}\right)\equiv0\pmod5.
		\end{align}
		With $n$ replaced by $pn$ in \eqref{e3.22}, we have
		\begin{align}\label{e3.24}
			c\displaystyle\left(p^8n+\frac{5(p^8-1)}{24}\right)\equiv c(n)\pmod5.
		\end{align}
	Thus, induction on \eqref{e3.24} yields for every integer $k\ge0$,
		\begin{align}\label{e3.25}
			c\displaystyle\left(p^{8k}n+\frac{5(p^{8k}-1)}{24}\right)\equiv c(n)\pmod5.
		\end{align}
Replacing $n$ by $p^7n+\frac{5(p^8-1)}{24}$ in \eqref{e3.25} and using \eqref{e3.23}, we obtain
		\begin{align}\label{e3.26}
			c\displaystyle\left(p^{8k+7}n+\frac{5(p^{8k+8}-1)}{24}\right)\equiv0\pmod5.
		\end{align}
		Replacing $n$ by $p^2n+\frac{5p(p^2-1)}{24}$ in \eqref{e3.20}, we get
		\begin{align}\label{e3.27}
			c\displaystyle\left(p^9n+\frac{5(p^{10}-1)}{24}\right)=\xi(p)c\displaystyle\left(p^7n+\frac{5(p^8-1)}{24}\right)-p\;c\displaystyle\left(p^5n+\frac{5(p^6-1)}{24}\right).
		\end{align}
		Now rewriting \eqref{e3.27}, by referring \eqref{e3.14} and \eqref{e3.21}, we obtain
		\begin{align}\label{e3.28}
			c\displaystyle\left(p^9n+\frac{5(p^{10}-1)}{24}\right)&=(\xi^4(p)-3p\xi^2(p)+p^2)\;c\displaystyle\left(pn+\frac{5(p^2-1)}{24}\right) \nonumber\\& \quad+(2p^2\xi(p)-p\xi^3(p))\;c\displaystyle\left(\frac{n}{p}\right).
		\end{align}
		Note that, if $p\equiv1\pmod5$ and $\xi(p) \equiv \pm2\pmod5$, or $p\equiv4\pmod5$ and $\xi(p) \equiv \pm1\pmod5$, then \eqref{e3.28} becomes
		\begin{align}\label{e3.29}
			c\displaystyle\left(p^9n+\frac{5(p^{10}-1)}{24}\right)\equiv s\; c\displaystyle\left(\frac{n}{p}\right)\pmod5,
		\end{align}
		where $s\not\equiv0\pmod 5$.\\
		If $p\nmid n$, then 
		\begin{align}\label{e3.30}
			c\displaystyle\left(p^9n+\frac{5(p^{10}-1)}{24}\right)\equiv0\pmod5.
		\end{align}
		Further, replacing $n$ by $pn$ in \eqref{e3.29}, we obtain
		\begin{align}\label{e3.31}
			c\displaystyle\left(p^{10}n+\frac{5(p^{10}-1)}{24}\right)\equiv s\; c(n)\pmod5.
		\end{align}
		Induction on \eqref{e3.31} gives for every integer $k\ge 0$, 
		\begin{align}\label{e3.32}
			c\displaystyle\left(p^{10k}n+\frac{5(p^{10k}-1)}{24}\right)\equiv s^k c(n)\pmod5.
		\end{align}
		Again, replacing $n$ by $p^9n+\frac{5(p^{10}-1)}{24}$ in \eqref{e3.32} and using \eqref{e3.30}, we obtain 
		\begin{align}\label{e3.33}
			c\displaystyle\left(p^{10k+9}n+\frac{5(p^{10k+10}-1)}{24}\right)\equiv0\pmod5.
		\end{align}
		Now, replacing $n$ by $p^2n+\frac{5p(p^2-1)}{24}$ in \eqref{e3.27}, we obtain
		\begin{align}\label{e3.34}
			c\displaystyle\left(p^{11}n+\frac{5(p^{12}-1)}{24}\right)=\xi(p)c\displaystyle\left(p^9n+\frac{5(p^
				{10}-1)}{24}\right)-p\;c\displaystyle\left(p^7n+\frac{5(p^8-1)}{24}\right).
		\end{align}
	Substituting \eqref{e3.21} and \eqref{e3.28} in RHS of \eqref{e3.34}, we get
		\begin{align}\label{e3.35}
			c\displaystyle\left(p^
			{11}n+\frac{5(p^{12}-1)}{24}\right)&=(\xi^5(p)-4p\xi^3(p)+3p^2\xi(p))\;c\displaystyle\left((pn+\frac{5(p^2-1)}{24}\right) \nonumber \\ & \quad+(3p^2\xi^2(p)-p\xi^4(p)-p^3)\;c\displaystyle\left(\frac{n}{p}\right).
		\end{align}
		In particular, for $p\equiv2\pmod5$ and $\xi(p) \equiv \pm1\pmod5$, or $p\equiv3\pmod5$ and $\xi(p) \equiv\pm2\pmod5$, we have from \eqref{e3.35} that
		\begin{align}\label{e3.36}
			c\displaystyle\left(p^
			{11}n+\frac{5(p^{12}-1)}{24}\right)\equiv r\; c\displaystyle\left(\frac{n}{p}\right)\pmod5,
		\end{align}
		where $r\equiv\pm2\pmod 5$.\\ Observe that, if $p\nmid n$, then
		\begin{align}\label{e3.37}
			c\displaystyle\left(p^{11}n+\frac{5(p^{12}-1)}{24}\right)\equiv0\pmod5.
		\end{align}
		Replacing $n$ by $pn$ in \eqref{e3.36}, we obtain
		\begin{align}\label{e3.38}
			c\displaystyle\left(p^{12}n+\frac{5(p^{12}-1)}{24}\right)\equiv r\; c(n)\pmod5.
		\end{align}
		By induction on \eqref{e3.38}, we see that for every integer $k\ge 0$,
		\begin{align}\label{e3.39}
			c\displaystyle\left(p^{12k}n+\frac{5(p^{12k}-1)}{24}\right)\equiv r^k c(n)\pmod5.
		\end{align}
		Now, replacing $n$ by $p^{11}n+\frac{5(p^{12}-1)}{24}$ in \eqref{e3.39} and using \eqref{e3.37}, we obtain 
		\begin{align}\label{e3.40}
			c\displaystyle\left(p^{12k+11}n+\frac{5(p^{12k+12}-1)}{24}\right)\equiv0\pmod5.
		\end{align}
		From \eqref{e1.2} and generating function \eqref{e3.1}, we have
		\begin{align}\label{e3.41}
			a_3(5n+1) \equiv 3c(n)\pmod 5.
		\end{align}
		Using \eqref{e3.12}, \eqref{e3.19}, \eqref{e3.26}, \eqref{e3.33}, \eqref{e3.40}, and \eqref{e3.41}, we settle the proof of congruence \eqref{e1.4}.
	\end{proof}
		\begin{proof}[Proof of (2) of Theorem \ref{t1.3}\unskip]
 We observe from, \eqref{e3.6}, if $\xi(p)\equiv\displaystyle\left(\frac{-2n+\frac{5(p^2-1)}{12}}{p}\right)_L\pmod 5$, then
		\begin{align}\label{e3.42}
			c\displaystyle\left(p^2n+\dfrac{5(p^2-1)}{24}\right)\equiv-p\;c\displaystyle\left(\frac{n-\frac{5(p^2-1)}{24}}{p^2}\right)\pmod 5.
		\end{align}
		We also observe that if $\displaystyle\left(\frac{-2n+\frac{5(p^2-1)}{12}}{p}\right)_L\not\equiv0\pmod 5$, then $p\nmid (24n+5)$. 	It is easy to verify $\frac{n-\frac{5(p^2-1)}{24}}{p^2}$ is not an integer, thus, we have
		\begin{align}\label{e3.43}
			c\displaystyle\left(\frac{n-\frac{5(p^2-1)}{24}}{p^2}\right)=0.
		\end{align}
\
		From \eqref{e3.42} and \eqref{e3.43}, we get
		\begin{align}\label{e3.44}
			c\displaystyle\left(p^2n+\dfrac{5(p^2-1)}{24}\right)\equiv0\pmod5.
		\end{align}
		Using \eqref{e3.44} in \eqref{e3.18}, \eqref{e3.25}, \eqref{e3.32}, \eqref{e3.39} with $n$ replaced by $p^2n+\frac{5(p^2-1)}{24}$, we settle the proof of congruence \eqref{e1.6}.
	\end{proof}
		\begin{proof}[Proof of (3) of Theorem \ref{t1.3}\unskip]
		If $p=5$, (it is clear that $\xi(p)=2$) then \eqref{e3.7} implies
		\begin{align}\label{e3.49}
			c\displaystyle\left(p^3n+\frac{5(p^4-1)}{24}\right)\equiv 2 c\displaystyle\left(pn+\frac{5(p^2-1)}{24}\right)\pmod5.
		\end{align}
		Replacing $n$ by $p^2n+\frac{5p(p^2-1)}{24}$ in \eqref{e3.49} and using induction, we obtain for $k\ge0$,
		\begin{align}\label{e3.50}
			c\displaystyle\left(p^{2(k+1)+1}n+\frac{5(p^{2(k+2)}-1)}{24}\right)\equiv (2)^{k+1}c\displaystyle\left(pn+\frac{5(p^2-1)}{24}\right)\pmod5.
		\end{align}
	Using \eqref{e3.50} in \eqref{e3.41}, we complete the proof of Theorem \ref{t1.3}.
	\end{proof} 	
	\begin{remark}
		For example, setting $p=11$ and $k=0$ in Theorem \eqref{t1.3}, we find that $\xi(11)\equiv\displaystyle\left(\frac{50-2n}{p}\right)_L\equiv1\pmod 5$ when $n\equiv1, 4, 6, 7, 8\pmod{11}$, and for $n\ge0$,
		\begin{align*}
			a_3(6655n+605j+126)\equiv0\pmod5,	
		\end{align*}
		where $j\in\{1, 4, 6, 7, 8\}$.
	\end{remark}
	\section{Congruences modulo 3 for $a_t(n)$, where $t\in\{5,8,11,14,17,20,23,26\}$} \label{s4}
In this section, all the congruences hold to the modulo 3. We define $\omega(p)$ purely for notational convenience:
	\begin{align}\label{e4.1}
		\omega(p) :=\begin{cases}
			\hphantom{1}4, & \text{if } \xi_i(p) \equiv 0\!\!\!\!\pmod 3\\[4pt]
			\hphantom{1}6, & \text{if }\xi_i(p) \not\equiv0\!\!\!\!\pmod3 \;\text{and}\; p\equiv1\!\!\!\!\pmod3\\[4pt]
			\hphantom{1}8, & \text{if } \xi_i(p) \not\equiv0\!\!\!\!\pmod3 \;\text{and}\;p\equiv2\!\!\!\!\pmod3\;,\\[4pt]
		\end{cases}
	\end{align}
where $i\in\{1,2,3,4,5,6,7,8,9,10,11,12\}$.
	\begin{theorem}\label{t4.1}
For $n$ be a positive integer, $p\ge5$ be a prime, and $\omega(p)$ be defined as in \eqref{e4.1}, we define
	\begin{align}\label{e4.2}
	\xi_1(p):= c_1\displaystyle\left(\frac{3(p^2-1)}{8}\right) +p \displaystyle\left(\frac{-\frac{3}{4}(p^2-1)}{p}\right)_L,
\end{align}
where 
		\begin{align}\label{e4.3}
			\sum_{n=0}^{\infty}c_1(n)q^n:=f_1f_2^4.
		\end{align}
		The following hold:\\
(1). For $k\ge0$, if $p\nmid n$, then
\begin{align}\label{e4.4}
			a_5\displaystyle\left(3p^{\omega(p)(k+1)-1}n + \dfrac{9p^{\omega(p)(k+1)}-1}{8}\right) \equiv0.
		\end{align}
		(2). If $\xi_1(p)\not\equiv0$, then for $k\ge0$ with $\xi_1(p)\equiv p\displaystyle\left(\frac{2n-\frac{3}{4}(p^2-1)}{p}\right)_L$,
		\begin{align}\label{e4.5}
			a_5\displaystyle\left(3p^{\omega(p)k+2}n + \dfrac{9p^{\omega(p)k+2}-1}{8}\right) \equiv0.
		\end{align}
	\end{theorem}
	\begin{proof}[Proof of (1) of Theorem \ref{t4.1}\unskip]
	Thanks to Lemma \ref{l2.2}, from \eqref{e4.3} we have
		\begin{align}\label{e4.6}
			c_1\displaystyle\left(p^2n+\dfrac{3(p^2-1)}{8}\right)=\displaystyle \left(p^3\alpha-p\displaystyle\left(\frac{2n-\frac{3(p^2-1)}{4}}{p}\right)_L\right)c_1(n)-p^3 \;c_1\displaystyle\left(\frac{n-\frac{3(p^2-1)}{8}}{p^2}\right).
		\end{align}
 Setting $n=0$ in \eqref{e4.6}, we get
		\begin{align}\label{e4.7}
			p^3\alpha=c_1\displaystyle\left(\frac{3(p^2-1)}{8}\right)+p\displaystyle\left(\frac{\frac{-3(p^2-1)}{4}}{p}\right)_L:=\xi_1(p).
		\end{align}
		Using \eqref{e4.7}, we rewrite \eqref{e4.6} as
		\begin{align}\label{e4.8}
			c_1\displaystyle\left(p^2n+\dfrac{3(p^2-1)}{8}\right)=\displaystyle\left(\xi_1(p)-p\displaystyle\left(\frac{2n-\frac{3(p^2-1)}{4}}{p}\right)_L\right)c_1(n)-p^3\;c_1\displaystyle\left(\frac{n-\frac{3(p^2-1)}{8}}{p^2}\right).
		\end{align}
		Replacing $n$ by $pn+\frac{3(p^2-1)}{8}$ in \eqref{e4.8}, we obtain
		\begin{align}\label{e4.9}
			c_1\displaystyle\left(p^3n+\frac{3(p^4-1)}{8}\right)=\xi_1(p)c_1\displaystyle\left(pn+\frac{3(p^2-1)}{8}\right)-p^3\;c_1\displaystyle\left(\frac{n}{p}\right).
		\end{align}
	For, if $\xi_1(p)\equiv0$ in \eqref{e4.9}, we have
		\begin{align}\label{e4.10}
			c_1\displaystyle\left(p^3n+\frac{3(p^4-1)}{8}\right)\equiv -p^3\;c_1\displaystyle\left(\frac{n}{p}\right).
		\end{align}
		Now, replacing $n$ by $pn$ in \eqref{e4.10}, we have
		\begin{align}\label{e4.11}
			c_1\displaystyle\left(p^4n+\frac{3(p^4-1)}{8}\right)\equiv-p^3\;c_1(n).	
		\end{align}
	For $k\ge0$, induction on \eqref{e4.11} gives
		\begin{align}\label{e4.12}
			c_1\displaystyle\left(p^{4k}n+\frac{3(p^{4k}-1)}{8}\right)\equiv(-p^3)^kc_1(n).	
		\end{align}
		Also, if $\xi_1(p)\equiv0$ and $p\nmid n$, then by \eqref{e4.10} we get
		\begin{align}\label{e4.13}
			c_1\displaystyle\left(p^3n+\frac{3(p^4-1)}{8}\right)\equiv0.
		\end{align}
		Further, in \eqref{e4.12} with $n$ replaced by $p^3n+\frac{3(p^4-1)}{8}$, we employ \eqref{e4.13} to obtain
		\begin{align}\label{e4.14}
			c_1\displaystyle\left(p^{4k+3}n+\frac{3(p^{4k+4}-1)}{8}\right)\equiv0.	
		\end{align}
	Again, in \eqref{e4.9} with $n$ replaced by $p^2n+\frac{3p(p^2-1)}{8}$ yields us
		\begin{align}\label{e4.15}
			c_1\displaystyle\left(p^5n+\frac{3(p^6-1)}{8}\right)=\xi_1(p)c_1\displaystyle\left(p^3n+\frac{3(p^4-1)}{8}\right)-p^3\;c_1\displaystyle\left(pn+\frac{3(p^2-1)}{8}\right).
		\end{align}
Using \eqref{e4.9} and \eqref{e4.15}, we have
		\begin{align}\label{e4.16}
			c_1\displaystyle\left(p^5n+\frac{3(p^6-1)}{8}\right)=(\xi_1^2(p)-p^3)c_1\displaystyle\left(pn+\frac{3(p^2-1)}{8}\right)-p^3\xi_1(p)\;c_1\displaystyle\left(\frac{n}{p}\right).
		\end{align}
If $\xi_1(p)\not \equiv 0$ and $p\equiv1$, then \eqref{e4.16} becomes
		\begin{align}\label{e4.17}
			c_1\displaystyle\left(p^5n+\frac{3(p^6-1)}{8}\right)\equiv-p^3\xi_1(p)\;c_1\displaystyle\left(\frac{n}{p}\right).
		\end{align}
		If $p\nmid n$ then $c_1\displaystyle\left(\frac{n}{p}\right)=0$. Thus, \eqref{e4.17} can be rewritten as
		\begin{align}\label{e4.18}
			c_1\displaystyle\left(p^5n+\frac{3(p^6-1)}{8}\right)\equiv0.
		\end{align}
	Using induction on \eqref{e4.17}, with $n$ replaced by $pn$, we immediately get
		\begin{align}\label{e4.19}
			c_1\displaystyle\left(p^{6k}n+\frac{3(p^{6k}-1)}{8}\right)\equiv(-p^3\xi_1(p))^kc_1(n).
		\end{align}
	Again, in \eqref{e4.19} with $n$ replaced by $p^3n+\frac{3(p^4-1)}{8}$, we employ \eqref{e4.18} to obtain 
		\begin{align}\label{e4.20}
			c_1\displaystyle\left(p^{6k+5}n+\frac{3(p^{6k+6}-1)}{8}\right)\equiv0.
		\end{align}
		Replacing $n$ by $p^2n+\frac{3p(p^2-1)}{8}$ in \eqref{e4.15}, gives
		\begin{align}\label{e4.21}
			c_1\displaystyle\left(p^7n+\frac{3(p^8-1)}{8}\right)=\xi_1(p)c_1\displaystyle\left(p^5n+\frac{3(p^6-1)}{8}\right)-p^3\;c_1\displaystyle\left(p^3n+\frac{3(p^4-1)}{8}\right).
		\end{align}
		In \eqref{e4.21}, substituting for $c_1\displaystyle\left(p^3n+\frac{3(p^4-1)}{8}\right)$ and $c_1\displaystyle\left(p^5n+\frac{3(p^6-1)}{8}\right)$ from \eqref{e4.9} and \eqref{e4.16} respectively, we have
		\begin{align}\label{e4.22}
			c_1\displaystyle\left(p^7n+\frac{3(p^8-1)}{8}\right)=(\xi_1^3(p)-2\xi_1(p)p^3)c_1\displaystyle\left(pn+\frac{3(p^2-1)}{8}\right)+(p^6-p^3\;\xi_1^2(p))\;c_1\displaystyle\left(\frac{n}{p}\right).
		\end{align}
		If $\xi_1(p)\not \equiv 0$ and $p\equiv2$, then \eqref{e4.22} becomes
		\begin{align}\label{e4.23}
			c_1\displaystyle\left(p^7n+\frac{3(p^8-1)}{8}\right)\equiv 2c_1\displaystyle\left(\frac{n}{p}\right).
		\end{align}
	Now, if $p\nmid n$, we have
		\begin{align}\label{e4.24}
			c_1\displaystyle\left(p^7n+\frac{3(p^8-1)}{8}\right)\equiv0.
		\end{align}
	Using induction on \eqref{e4.24}, with $n$ replaced by $pn$, we immediately get
		\begin{align}\label{e4.25}
			c_1\displaystyle\left(p^{8k}n+\frac{3(p^{8k}-1)}{8}\right)\equiv (-1)^kc_1(n).
		\end{align}
		Replacing $n$ by $p^7n+\frac{3(p^8-1)}{8}$ in \eqref{e4.25} and using \eqref{e4.24}, we obtain
		\begin{align}\label{e4.26}
			c_1\displaystyle\left(p^{8k+7}n+\frac{3(p^{8k+8}-1)}{8}\right)\equiv0.
		\end{align}
Again, thanks to the generating function of Sellers \eqref{e2.1} and definition of $c_1(n)$, we have
		\begin{align}\label{e4.27}
			a_5(3n+1) \equiv 2c_1(n).
		\end{align}
		Using \eqref{e4.14}, \eqref{e4.20}, \eqref{e4.26}, and \eqref{e4.27}, we settle the proof of \eqref{e4.4}.
\end{proof}	
\begin{proof}[Proof of (2) of Theorem \ref{t4.1}\unskip]
We observe \eqref{e4.8} gives 
		\begin{align}\label{e4.28}
			c_1\displaystyle\left(p^2n+\dfrac{3(p^2-1)}{8}\right)\equiv-p^3\;c_1\displaystyle\left(\frac{n-\frac{3(p^2-1)}{8}}{p^2}\right),
		\end{align}
		when $\xi_1(p)\equiv p\displaystyle\left(\frac{2n-\frac{3(p^2-1)}{4}}{p}\right)_L$.\\
		Again, if $\displaystyle\left(\frac{2n-\frac{3(p^2-1)}{4}}{p}\right)_L\not\equiv0$, then $p\nmid (8n+3)$, which implies $\frac{n-\frac{3(p^2-1)}{8}}{p^2}$ is non-integer. Thus, we have
		\begin{align}\label{e4.29}
			c_1\displaystyle\left(\frac{n-\frac{3(p^2-1)}{8}}{p^2}\right)=0.
		\end{align}
		From \eqref{e4.28} and \eqref{e4.29}, we get
		\begin{align}\label{e4.30}
			c_1\displaystyle\left(p^2n+\dfrac{3(p^2-1)}{8}\right)\equiv0.
		\end{align}
		Replacing $n$ by $p^2n+\frac{3(p^2-1)}{8}$ in \eqref{e4.19} and using \eqref{e4.30}, we get
		\begin{align}\label{e4.31}
			c_1\displaystyle\left(p^{6k+2}n+\frac{3(p^{6k+2}-1)}{8}\right)\equiv0 .
		\end{align}
		Further, in \eqref{e4.25} with $n$ replaced by $p^2n+\frac{3(p^2-1)}{8}$, we employ \eqref{e4.30} to obtain
		\begin{align}\label{e4.32}
			c_1\displaystyle\left(p^{8k+2}n+\frac{3(p^{8k+2}-1)}{8}\right)\equiv 0.
		\end{align}
This settles the proof of \eqref{e4.5}.
	\end{proof} 	
	In the following theorem, we state a new family of congruences for $a_8(n)$.
	\begin{theorem}\label{t3.2}
		For $n$ be a positive integer, $p\ge5$ be a prime, and $\omega(p)$ be defined as in \eqref{e4.1}, we define
		\begin{align*}
		\xi_2(p):= c_2\displaystyle\left(\frac{p^2-1}{12}\right) +p\displaystyle\left(\frac{2^{-2}}{p}\right)_L \displaystyle\left(\frac{\frac{-(p^2-1)}{12}}{p}\right)_L,
	\end{align*}
		where 
			\begin{align*}
			\sum_{n=0}^{\infty}c_2(n)q^n:=\dfrac{f_1^8}{f_2^3}.
		\end{align*}
		The following hold:\\
		1. For $k\ge0$, if $p\nmid n$, then
		\begin{align*}
			a_8\displaystyle\left(3p^{\omega(p)(k+1)-1}n + \dfrac{p^{\omega(p)(k+1)}-1}{4}\right) \equiv0.
		\end{align*}
		2. If $\xi_2(p)\not\equiv0$, then for $k\ge0$ with $\xi_2(p)\equiv p\displaystyle\left(\frac{2^{-2}}{p}\right)_L\displaystyle\left(\frac{n-\frac{p^2-1}{12}}{p}\right)_L$,
		\begin{align*}
			a_8\displaystyle\left(3p^{\omega(p)k+2}n + \dfrac{p^{\omega(p)k+2}-1}{4}\right) \equiv0.
		\end{align*}
	\end{theorem}
		In the next theorem, we state a new family of congruences for $a_8(n)$.
		\begin{theorem}
			For $n$ be a positive integer, $p\ge5$ be a prime, and $\omega(p)$ be defined as in \eqref{e4.1}, we define
			\begin{align*}
			\xi_3(p):= c_3\displaystyle\left(\frac{p^2-1}{4}\right) +p^2 \displaystyle\left(\frac{\frac{(p^2-1)}{4}}{p}\right)_L,
		\end{align*}
		where 
			\begin{align*}
			\sum_{n=0}^{\infty}c_3(n)q^n:=\dfrac{f_1^8}{f_2}.
		\end{align*}
		The following hold:\\
		1. For $k\ge0$, if $p\nmid n$, then
	\begin{align*}
		a_8\displaystyle\left(9p^{\omega(p)(k+1)-1}n + \dfrac{9p^{\omega(p)(k+1)}-1}{4}\right) \equiv0.
	\end{align*}
		2. If $\xi_3(p)\not\equiv0$, then for $k\ge0$ with $\xi_3(p)\equiv p^2\displaystyle\left(\frac{-n+\frac{(p^2-1)}{4}}{p}\right)_L$,
			\begin{align*}
			a_8\displaystyle\left(9p^{\omega(p)k+2}n + \dfrac{9p^{\omega(p)k+2}-1}{4}\right) \equiv0.
		\end{align*}
	\end{theorem}
		In the following theorem, we establish a new family of congruences for $a_{11}(n)$.
		\begin{theorem}
				For $n$ be a positive integer, $p\ge5$ be a prime, and $\omega(p)$ be defined as in \eqref{e4.1}, we define
				\begin{align*}
				\xi_4(p):= c_4\displaystyle\left(\frac{p^2-1}{8}\right) +p\displaystyle\left(\frac{2^{-1}}{p}\right)_L \displaystyle\left(\frac{-\frac{(p^2-1)}{8}}{p}\right)_L,
			\end{align*}
			where 
				\begin{align*}
				\sum_{n=0}^{\infty}c_4(n)q^n:=\dfrac{f_1^7}{f_2^2}.
			\end{align*}
			The following hold:\\
			1. For $k\ge0$, if $p\nmid n$, then
			\begin{align*}
			a_{11}\displaystyle\left(3p^{\omega(p)(k+1)-1}n + \dfrac{3p^{\omega(p)(k+1)}-3}{8}\right) \equiv0.
		\end{align*}
			2. If $\xi_4(p)\not\equiv0$, then for $k\ge0$ with $\xi_4(p)\equiv p\displaystyle\left(\frac{2^{-1}}{p}\right)_L\displaystyle\left(\frac{n-\frac{(p^2-1)}{8}}{p}\right)_L$,
				\begin{align*}
				a_{11}\displaystyle\left(3p^{\omega(p)k+2}n + \dfrac{3p^{\omega(p)k+2}-3}{8}\right) \equiv0.
			\end{align*}
	\end{theorem}
	In the theorem below, we derive a new family of congruences for $a_{11}(n)$.
		\begin{theorem}
				For $n$ be a positive integer, $p\ge5$ be a prime, and $\omega(p)$ be defined as in \eqref{e4.1}, we define
			\begin{align*}
			\xi_5(p):= c_5\displaystyle\left(\frac{11(p^2-1)}{24}\right) +p \displaystyle\left(\frac{-\frac{11}{12}(p^2-1)}{p}\right)_L,
		\end{align*}
			where 
			\begin{align*}
			\sum_{n=0}^{\infty}c_5(n)q^n:=\dfrac{f_2^6}{f_1}.
		\end{align*}
			The following hold:\\
			1. For $k\ge0$, if $p\nmid n$, then
			\begin{align*}
			a_{11}\displaystyle\left(3p^{\omega(p)(k+1)-1}n + \dfrac{11p^{\omega(p)(k+1)}-3}{8}\right) \equiv0.
		\end{align*}
			2. If $\xi_5(p)\not\equiv0$, then for $k\ge0$ with $\xi_5(p)\equiv p\displaystyle\left(\frac{2n-\frac{11(p^2-1)}{12}}{p}\right)_L$,
			\begin{align*}
			a_{11}\displaystyle\left(3p^{\omega(p)k+2}n + \dfrac{11p^{\omega(p)k+2}-3}{8}\right) \equiv0.
		\end{align*}
	\end{theorem}
	In the next theorem, we present a new family of congruences for $a_{14}(n)$.
		\begin{theorem}
				For $n$ be a positive integer, $p\ge5$ be a prime, and $\omega(p)$ be defined as in \eqref{e4.1}, we define
	\begin{align*}
	\xi_6(p):= c_6\displaystyle\left(\frac{p^2-1}{6}\right) +p \displaystyle\left(\frac{\frac{-(p^2-1)}{6}}{p}\right)_L,
\end{align*}
			where 
			\begin{align*}
				\sum_{n=0}^{\infty}c_6(n)q^n:=\dfrac{f_1^6}{f_2}.
			\end{align*}
			The following hold:\\
			1. For $k\ge0$, if $p\nmid n$, then
		\begin{align*}
			a_{14}\displaystyle\left(3p^{\omega(p)(k+1)-1}n + \dfrac{p^{\omega(p)(k+1)}-1}{2}\right) \equiv0.
		\end{align*}
			2. If $\xi_6(p)\not\equiv0$, then for $k\ge0$ with $\xi_6(p)\equiv p\displaystyle\left(\frac{n-\frac{(p^2-1)}{6}}{p}\right)_L$,
			\begin{align*}
			a_{14}\displaystyle\left(3p^{\omega(p)k+2}n + \dfrac{p^{\omega(p)k+2}-1}{2}\right) \equiv0.
		\end{align*}
	\end{theorem}
	
	The theorem below presents a new family of congruences satisfied by $a_{14}(n)$.
	\begin{theorem}
			For $n$ be a positive integer, $p\ge5$ be a prime, and $\omega(p)$ be defined as in \eqref{e4.1}, we define
		\begin{align*}
		\xi_7(p):= c_7\displaystyle\left(\frac{p^2-1}{2}\right) +p \displaystyle\left(\frac{\frac{-(p^2-1)}{2}}{p}\right)_L,
	\end{align*}
		where 
			\begin{align*}
			\sum_{n=0}^{\infty}c_7(n)q^n:=\dfrac{f_2^7}{f_1^2}.
		\end{align*}
		The following hold:\\
		1. For $k\ge0$, if $p\nmid n$, then
		\begin{align*}
		a_{14}\displaystyle\left(3p^{\omega(p)(k+1)-1}n + \dfrac{3p^{\omega(p)(k+1)}-1}{2}\right) \equiv0.
	\end{align*}
		2. If $\xi_7(p)\not\equiv0$, then for $k\ge0$ with $\xi_7(p)\equiv p\displaystyle\left(\frac{n-\frac{(p^2-1)}{2}}{p}\right)_L$,
	\begin{align*}
		a_{14}\displaystyle\left(3p^{\omega(p)k+2}n + \dfrac{3p^{\omega(p)k+2}-1}{2}\right) \equiv0.
	\end{align*}
	\end{theorem}
		In the following theorem, we state a new family of congruences for $a_{17}(n)$.
	\begin{theorem}
			For $n$ be a positive integer, $p\ge5$ be a prime, and $\omega(p)$ be defined as in \eqref{e4.1}, we define
		\begin{align*}
		\xi_8(p):= c_8\displaystyle\left(\frac{13(p^2-1)}{24}\right) +p \displaystyle\left(\frac{-\frac{13}{12}(p^2-1)}{p}\right)_L,
	\end{align*}
		where 
			\begin{align*}
			\sum_{n=0}^{\infty}c_8(n)q^n:=\dfrac{f_2^8}{f_1^3}.
		\end{align*}
		The following hold:\\
		1. For $k\ge0$, if $p\nmid n$, then
			\begin{align*}
			a_{17}\displaystyle\left(3p^{\omega(p)(k+1)-1}n + \dfrac{13p^{\omega(p)(k+1)}-5}{8}\right) \equiv0.
		\end{align*}
		2. If $\xi_8(p)\not\equiv0$, then for $k\ge0$ with $\xi_8(p)\equiv p\displaystyle\left(\frac{2n-\frac{13}{12}(p^2-1)}{p}\right)_L$,
	\begin{align*}
		a_{17}\displaystyle\left(3p^{\omega(p)k+2}n + \dfrac{13p^{\omega(p)k+2}-5}{8}\right) \equiv0.
	\end{align*}
	\end{theorem}
		In the next theorem, we state a new family of congruences for $a_{17}(n)$.
		\begin{theorem}
				For $n$ be a positive integer, $p\ge5$ be a prime, and $\omega(p)$ be defined as in \eqref{e4.1}, we define
			\begin{align*}
			\xi_9(p):= c_9\displaystyle\left(\frac{5(p^2-1)}{8}\right) +p^2 \displaystyle\left(\frac{\frac{5}{4}(p^2-1)}{p}\right)_L,
		\end{align*}
			where 
		\begin{align*}
			\sum_{n=0}^{\infty}c_9(n)q^n:=\dfrac{f_2^8}{f_1}.
		\end{align*}
			The following hold:\\
			1. For $k\ge0$, if $p\nmid n$, then
				\begin{align*}
				a_{17}\displaystyle\left(9p^{\omega(p)(k+1)-1}n + \dfrac{45p^{\omega(p)(k+1)}-5}{8}\right) \equiv0.
			\end{align*}
			2. If $\xi_9(p)\not\equiv0$, then for $k\ge0$ with $\xi_9(p)\equiv p^2\displaystyle\left(\frac{-2n+\frac{5}{4}(p^2-1)}{p}\right)_L$,
			\begin{align*}
			a_{17}\displaystyle\left(9p^{\omega(p)k+2}n + \dfrac{45p^{\omega(p)k+2}-5}{8}\right) \equiv0.
		\end{align*}
	\end{theorem}
		In the following theorem, we establish a new family of congruences for $a_{20}(n)$.
	\begin{theorem}
			For $n$ be a positive integer, $p\ge5$ be a prime, and $\omega(p)$ be defined as in \eqref{e4.1}, we define
			\begin{align*}
			\xi_{10}(p):= c_{10}\displaystyle\left(\frac{p^2-1}{4}\right) +p \displaystyle\left(\frac{\frac{-(p^2-1)}{4}}{p}\right)_L,
		\end{align*}
		where 
		\begin{align*}
			\sum_{n=0}^{\infty}c_{10}(n)q^n:=f_1^4f_2.
		\end{align*}
		The following hold:\\
		1. For $k\ge0$, if $p\nmid n$, then
		\begin{align*}
			a_{20}\displaystyle\left(3p^{\omega(p)(k+1)-1}n + \dfrac{3p^{\omega(p)(k+1)}-3}{4}\right) \equiv0.
		\end{align*}
		2. If $\xi_{10}(p)\not\equiv0$, then for $k\ge0$ with $\xi_{10}(p)\equiv p\displaystyle\left(\frac{n-\frac{(p^2-1)}{4}}{p}\right)_L$,
			\begin{align*}
			a_{20}\displaystyle\left(3p^{\omega(p)k+2}n + \dfrac{3p^{\omega(p)k+2}-3}{4}\right) \equiv0.
		\end{align*}
	\end{theorem}
	In the theorem below, we derive a new family of congruences for $a_{23}(n)$.
\begin{theorem}
		For $n$ be a positive integer, $p\ge5$ be a prime, and $\omega(p)$ be defined as in \eqref{e4.1}, we define
		\begin{align*}
		\xi_{11}(p):= c_{11}\displaystyle\left(\frac{5(p^2-1)}{8}\right) +p \displaystyle\left(\frac{\frac{-5(p^2-1)}{4}}{p}\right)_L,
	\end{align*}
	where 
	\begin{align*}
		\sum_{n=0}^{\infty}c_{11}(n)q^n:=\dfrac{f_2^{10}}{f_1^5}.
	\end{align*}
	The following hold:\\
	1. For $k\ge0$, if $p\nmid n$, then
\begin{align*}
	a_{23}\displaystyle\left(3p^{\omega(p)(k+1)-1}n + \dfrac{15p^{\omega(p)(k+1)}-7}{8}\right) \equiv0.
\end{align*}
	2. If $\xi_{11}(p)\not\equiv0$, then for $k\ge0$ with $\xi_{11}(p)\equiv p\displaystyle\left(\frac{2n-\frac{5}{4}(p^2-1)}{p}\right)_L$,
	\begin{align*}
		a_{23}\displaystyle\left(3p^{\omega(p)k+2}n + \dfrac{15p^{\omega(p)k+2}-7}{8}\right) \equiv0.
	\end{align*}
\end{theorem}
	In the next theorem, we establish a new family of congruences for $a_{26}(n)$.
\begin{theorem}\label{t3.12}
		For $n$ be a positive integer, $p\ge5$ be a prime, and $\omega(p)$ be defined as in \eqref{e4.1}, we define
	\begin{align*}
	\xi_{12}(p):= c_{12}\displaystyle\left(\frac{(p^2-1)}{3}\right) +p \displaystyle\left(\frac{\frac{-(p^2-1)}{3}}{p}\right)_L,
\end{align*}
	where 
	\begin{align*}
		\sum_{n=0}^{\infty}c_{12}(n)q^n:=f_1^2f_2^3.
	\end{align*}
	The following hold:\\
	1. For $k\ge0$, if $p\nmid n$, then
		\begin{align*}
		a_{26}\displaystyle\left(3p^{\omega(p)(k+1)-1}n +p^{\omega(p)(k+1)}-1\right) \equiv0.
	\end{align*}
	2. If $\xi_{12}(p)\not\equiv0$, then for $k\ge0$ with $\xi_{12}(p)\equiv p\displaystyle\left(\frac{n-\frac{(p^2-1)}{3}}{p}\right)_L$,
	\begin{align*}
	a_{26}\displaystyle\left(3p^{\omega(p)k+2}n +p^{\omega(p)k+2}-1\right) \equiv0.
\end{align*}
\end{theorem}
\begin{proof}[Proof of Theorem \ref{t3.2}-\ref{t3.12}\unskip] Thanks to Newman's result Lemma \ref{l2.2} and the generating functions \eqref{e2.2}-\eqref{e2.12}, respectively, we skip the detailed proof of Theorem \ref{t3.2}-\ref{t3.12}.
\end{proof}
\begin{remark}
For example, it is easy to check that $\xi_{12}(5)\equiv0$. Using (1) of Theorem \ref{t3.12}, we find that
\begin{align*}
a_{26}(1875n+375j+624)\equiv 0,
\end{align*}
when $j\in$\{1,2,3,4\}.\\
Also, we can verify that $\xi_{12}(7)\equiv2\pmod3$ and $p\displaystyle\left(\frac{n-\frac{(p^2-1)}{3}}{p}\right)_L\equiv 2$ for $p=7$ when $n\equiv 0,5\pmod7$. Therefore,
\begin{align*}
	a_{26}(1029n+147l+48)\equiv 0,
\end{align*}
when $l\in\{0,5\}$.
\end{remark}
\section{Congruences for modulo 3 and 5 for $a_5(n)$}\label{s5}
In this section, to prove our results we employ theory of modular forms.
\begin{theorem}\label{t1.4} For every $n,\alpha\ge0$, we have
	\begin{align}\label{e1.8}
		a_5\displaystyle\left(3^{2\alpha+3}n+\frac{153\cdot3^{2\alpha}-1)}{8}\right)\equiv0\pmod3.
	\end{align}
\end{theorem}
\begin{proof}[Proof of Theorem \ref{t1.4}\unskip] From the generating function \eqref{e1.1}, we have 
	\begin{align}\label{e5.1}
		\sum_{n=0}^{\infty}a_5(n)q^n= \frac{f_2^{4}}{f_1^5}.
	\end{align}
	We first prove a congruence modulo 3 for $a_5(n)$, corresponding to the $\alpha=0$ case of \eqref{e1.8}. Let 
	\begin{align*}
		f_1(z):= \frac{\eta^4(2z)}{\eta^5(z)}\eta^{189}(z).
	\end{align*}
	The eta-quotient defined above meets the criteria of Theorem \ref{t2.1} and hence, is also holomorphic at all cusps of $\Gamma_0(4)$. Thus, $f_1(z)\in M_{94}(\Gamma_0(4), \chi )$ with character $\chi=\displaystyle\left(\tfrac{-2^4}{\bullet}\right)$. Furthermore, by \eqref{e5.1} and \eqref{e2.1}, the Fourier expansion of $f_1(z)$ satisfies
	\begin{align*}
		f_1(z)=\displaystyle\left(\sum_{n=0}^{\infty}a_5(n)q^{n+8}\right)\prod_{n=1}^\infty(1-q^{n})^{189}.
	\end{align*}
	Using proposition \ref{p2.4} and \ref{p2.5}, we obtain
	\begin{align*}
		(f_1\mid T_3^3)(z)\equiv\displaystyle\left(\sum_{n=0}^{\infty}a_5(27n+19)q^{n+1}\right)\prod_{n=1}^\infty(1-q^{n})^{7}\pmod 3.
	\end{align*}
	Since the weight of $f_1(z)$ is 94 and level is 4, the Sturm's bound for the associated space of modular forms is found to be 47. We use Mathematica to confirm that the Fourier coefficients of $f_1(z)\mid T_3^3$, up to this bound is congruent to 0 modulo 3. Therefore, using Theorem \ref{t2.2} we conclude that
	\begin{align}\label{e4.2}
		a_5\displaystyle\left(27n+19\right)\equiv0\pmod3,
	\end{align} 
	for all $n\ge0$.\\
	\indent Next, we construct 
	\begin{align*}
		h_2(z):=\frac{\eta^4(2z)}{\eta^5(z)}\eta^{189}(z)E^{189}_4= f_1(z)E^{189}_4
	\end{align*}
	and
	\begin{align*}
		h_3(z):=\frac{\eta^4(2z)}{\eta^5(z)}\eta^{1701}(z).
	\end{align*}
	Clearly, $E_4(z) \in M_4(\Gamma_0(1))$. It is easy to verify both $h_2(z)$ and $h_3(z)$ belongs to the same space $ M_{850}(\Gamma_0(4),\chi) $ with character  $\chi=\displaystyle\left(\tfrac{-2^4}{\bullet}\right)$.
	Since  $E_4(z)\equiv 1\pmod3$, the Fourier expansions of these two forms satisfy 
	\begin{align*}
		h_2(z)\equiv\displaystyle\left(\sum_{n=0}^{\infty}a_5(n)q^{n+8}\right)\prod_{n=1}^\infty(1-q^{n})^{189}\pmod 3
	\end{align*}
	and
	\begin{align*}
		h_3(z)\equiv\displaystyle\left(\sum_{n=0}^{\infty}a_5(n)q^{n+71}\right)\prod_{n=1}^\infty(1-q^{n})^{1701}\pmod 3.
	\end{align*}
	Now, applying $T_3$ operator twice to $h_2(z)$ and four times to $h_3(z)$, we obtain
	\begin{align*}
		(h_2\mid T_3^2)(z)\equiv\displaystyle\left(\sum_{n=0}^{\infty}a_5(9n+1)q^{n+1}\right)\prod_{n=1}^\infty(1-q^{n})^{21}\pmod 3
	\end{align*}
	and
	\begin{align*}
		(h_3\mid T_3^4)(z)\equiv\displaystyle\left(\sum_{n=0}^{\infty}a_5(81n+10)q^{n+1}\right)\prod_{n=1}^\infty(1-q^{n})^{21}\pmod 3.
	\end{align*}
	Using Mathematica, we verify that the Fourier coefficients of these two expansions agree modulo 3 up to Sturm's bound of 425. By Theorem \ref{t2.2}, we conclude that $	(h_2(z)\mid T_3^2)(z)\equiv (h_3(z)\mid T_3^4)(z)\pmod 3$. Thus,
	\begin{align*}
		a_5\displaystyle\left(9n+1\right)\equiv a_5\displaystyle\left(81n+10\right)\pmod3,
	\end{align*}
	for all $n\ge0$. By induction on above self similarity congruences and \eqref{e4.2}, we complete the proof of Theorem \ref{t1.4}. 
\end{proof}
\begin{theorem}\label{t1.5}
	For every non-negative integer n, we have
	\begin{align}\label{e1.9}
		a_5\displaystyle\left(5n+3\right)\equiv0\pmod5.
	\end{align}
\end{theorem}
\begin{proof}[Proof of Theorem \ref{t1.5}\unskip]
	Define
	\begin{align*}
		f_2(z):= \frac{\eta^4(2z)}{\eta^5(z)}\eta^{45}(z).
	\end{align*}
	Clearly, $f_2(z)\in M_{22}(\Gamma_0(4), \chi )$ with character $\chi=\displaystyle\left(\tfrac{-2^4}{\bullet}\right)$.\\
	\indent Moreover, by generating function \eqref{e5.1} and \eqref{e2.1} the Fourier expansion of $f_2(z)$ satisfies
	\begin{align*}
		f_2(z)=\displaystyle\left(\sum_{n=0}^{\infty}a_5(n)q^{n+2}\right)\prod_{n=1}^\infty(1-q^{n})^{45}.
	\end{align*}
	Applying $T_5$ operator to $f_2(z)$, we obtain
	\begin{align*}
		(f_2\mid T_5)(z)=\displaystyle\left(\sum_{n=0}^{\infty}a_5(5n+3)q^{n+1}\right)\prod_{n=1}^\infty(1-q^{n})^{9}\pmod 5.
	\end{align*}
	Using Mathematica, we verify that the Fourier coefficients of $f_2(z)\mid T_5$ up to Sturm's bound of 11 is congruent to 0 modulo 5. Hence, by Theorem \ref{t2.2} we can conclude that
	\begin{align*}
		a_5\displaystyle\left(5n+3\right)\equiv0\pmod5,
	\end{align*}
	for all $n\ge0$.
\end{proof}

\noindent\textbf{Acknowledgment.} The authors gratefully acknowledge Professor James A. Sellers for beneficial conversation during the preparation of this paper.
	
	\bigskip

	\noindent
	Department of Mathematics\\
	Ramanujan School of Mathematical Sciences\\
	Pondicherry University\\
	Puducherry- 605 014, India.\\
	\noindent Email: \texttt{tthejithamp@pondiuni.ac.in}
	
	\noindent	Email: \texttt{dr.fathima.sn@pondiuni.ac.in} (\Letter)


\begin{thebibliography}{99}
	\bibitem{am1}Amdeberhan, T., and Merca, M. (2025). From crank to congruences. arXiv preprint arXiv:2505.19991.
\bibitem{andr}Andrews, G. E. (1998). The theory of partitions (No. 2). Cambridge university press.

\bibitem{cl1} Corteel, S., and Lovejoy, J. (2004). Overpartitions. Transactions of the American Mathematical Society, 356(4), 1623-1635.
	\bibitem{gh}Gordon, B.,and Hughes, K. (2006, October). Ramanujan congruences for q (n). In Analytic Number Theory: Proceedings of a Conference Held at Temple University, Philadelphia, May 12–15, 1980 (pp. 333-359). Berlin, Heidelberg: Springer Berlin Heidelberg.


	
	\bibitem{rg1}Guadalupe, R. (2025). A note on congruences for the difference between even cranks and odd cranks. arXiv preprint arXiv:2506.16267.

		\bibitem{hs1}Hirschhorn, M. D., and Sellers, J. A. (2025). A Family of Congruences Modulo 7 for Partitions with Monochromatic Even Parts and Multi--Colored Odd Parts. arXiv preprint arXiv:2507.09752.
		
		\bibitem{knop}Knopp, M. I. (2008). Modular functions in analytic number theory (Vol. 337). American Mathematical Soc.
		\bibitem{kmf}Koblitz, N. (1993). Introduction to Elliptic Curves and Modular Forms (Vol. 97). Springer Science and Business Media.
		
		\bibitem{lig}Ligozat, G. (1975). Courbes modulaires de genre 1 (No. 43). Société mathématique de France.
		
		
		\bibitem{new 2}Newman, M. (1959). Construction and application of a class of modular functions (II). Proceedings of the London Mathematical Society, 3(3), 373-387.
		
		\bibitem{new62}Newman, M. (1962). Modular forms whose coefficients possess multiplicative properties, II. Annals of Mathematics, 75(2), 242-250.
			\bibitem{owom}Ono, K. (2004). The Web of Modularity: Arithmetic of the Coefficients of Modular Forms and $ q $-series: Arithmetic of the Coefficients of Modular Forms and Q-series (No. 102). American Mathematical Soc..
			\bibitem{rmf}Rankin, R. A. (1977). Modular forms and functions. Cambridge University Press.
		\bibitem{sellers}Sellers, J. A. (2025). Elementary Proofs and Generalizations of Recent Congruences of Thejitha and Fathima. arXiv preprint arXiv:2510.01572.
			\bibitem{onen}Sloane,N.J.A., et al.: The on-line encyclopedia of integer sequences (2023).https://oeis.org
		
		\bibitem{sm1}Smoot, N. A. (2021). On the computation of identities relating partition numbers in arithmetic progressions with eta quotients: an implementation of Radu's algorithm. Journal of Symbolic Computation, 104, 276-311.
		\bibitem{sb}Sturm, J. (2006, September). On the congruence of modular forms. In Number Theory: A Seminar held at the Graduate School and University Center of the City University of New York 1984–85 (pp. 275-280). Berlin, Heidelberg: Springer Berlin Heidelberg.
		
		
	
	\end{thebibliography}
\end{document}